\documentclass[final]{amsart}

\usepackage[english]{babel}
\usepackage[utf8]{inputenc}

\usepackage{amsfonts,amsthm,amssymb,amsmath,amsthm,latexsym,wasysym,stmaryrd,mathrsfs,dsfont,txfonts}
\usepackage{accents,multirow,rotating,subfigure,graphicx,bigstrut,lscape,multicol,fancyhdr,enumerate,accents,mathtools,exscale}

\usepackage{pdftricks}
\usepackage{color}
\usepackage[pdftex,all]{xy}

\usepackage{hyperref}
\usepackage{appendix}
\usepackage[justification=centering]{caption}

\graphicspath{{Figures/}}

\theoremstyle{theorem}
\newtheorem {theo}{Theorem}[section]
\newtheorem*{theo*}{Theorem}
\newtheorem {lemme}[theo]{Lemma}
\newtheorem*{lemme*}{Lemma}

\newtheorem*{prop*}{Proposition}

\newtheorem*{cor*}{Corollary}

\newtheorem*{cor_proof*}{Corollary (of the proof)}

\newtheorem*{conjecture*}{Conjecture}
\theoremstyle{definition}
\newtheorem {defi}[theo]{Definition}
\newtheorem*{defi*}{Definition}

\newtheorem*{nota*}{Notation}
\theoremstyle{remark}
\newtheorem {remarque}[theo]{Remark}
\newtheorem*{remarque*}{Remark}

\newtheorem*{warning*}{Warning}

\newtheorem*{remarques*}{Remarks}
\newtheorem {warnings}[theo]{Warnings}
\newtheorem*{warnings*}{Warnings}

\newtheorem*{convention*}{Convention}
\newtheorem {exemple}[theo]{Example}
\newtheorem*{exemple*}{Example}

\newtheorem*{exemples*}{Examples}
\newtheorem {question}[theo]{Question}
\newtheorem*{question*}{Question}

\newtheorem*{questions*}{Questions}

\newtheorem*{fact*}{Fact}
\newtheorem*{acknowledgments}{Acknowledgments}

\def\DD{{\mathcal D}}
\def\N{{\mathds N}}
\def\Z{{\mathds Z}}

\def\p{\partial}
\def\UnN{\{1,\cdots,n\}}
\def\AA{\mathcal{A}}

\def\ie{{\it i.e. }}

\def\pcup{\operatornamewithlimits{\cup}\limits}
\def\psqcup{\operatornamewithlimits{\sqcup}\limits}

\def\fract#1/#2{\hbox{\leavevmode
  \kern.1em \raise .25ex \hbox{\the\scriptfont0 $#1$}\kern-.1em }\big/
  {\hbox{\kern-.15em \lower .5ex \hbox{\the\scriptfont0 $#2$}} }}

\def\fractt#1/#2{\hbox{\leavevmode
  \kern.1em \raise .25ex \hbox{\the\scriptfont0 $#1$}\kern-.1em
}\lower .2ex\hbox{\Big/}
  {\hbox{\kern-.15em \lower .8ex \hbox{\the\scriptfont0 $#2$}} }}

\def\subfract#1/#2{\hbox{\leavevmode
  \kern.1em \raise .25ex \hbox{\the\scriptfont0 \scriptsize $#1$}\kern-.1em }/
  {\hbox{\kern-.15em \lower .5ex \hbox{\the\scriptfont0 \scriptsize $#2$}} }}

\newcommand{\noi}{\noindent}

\newcommand{\dessin}[2]{
  \vcenter{\hbox{\includegraphics[height=#1]{#2.pdf}}}}

\def\sv{\textrm{sv}}

\def\cc{\textrm{sc}}

\def\TC{\textrm{TC}}
\def\OC{\textrm{OC}}
\def\UC{\textrm{UC}}

\def\SL{\textrm{SL}}

\def\SLn{\SL_n}

\def\SLhn{\SLn^{\cc}}

\def\P{\textrm{P}}

\def\Ph{\P^{\cc}}
\def\Pn{\P_n}

\def\Phn{\Pn^{\cc}}

\def\wSL{\textrm{wSL}}
\def\wSLH{\wSL^{\sv}}
\def\wSLh{\wSL^{\cc}}
\def\wSLn{\wSL_n}
\def\wSLHn{\wSLn^{\sv}}
\def\wSLhn{\wSLn^{\cc}}

\def\wP{\textrm{wP}}
\def\wPH{\wP^{\sv}}
\def\wPh{\wP^{\cc}}
\def\wPn{\wP_n}
\def\wPHn{\wPn^{\sv}}
\def\wPhn{\wPn^{\cc}}

\def\vSL{\textrm{vSL}}
\def\vSLH{\vSL^{\sv}}
\def\vSLh{\vSL^{\cc}}
\def\vSLn{\vSL_n}
\def\vSLHn{\vSLn^{\sv}}
\def\vSLhn{\vSLn^{\cc}}

\def\vP{\textrm{vP}}
\def\vPH{\vP^{\sv}}
\def\vPh{\vP^{\cc}}
\def\vPn{\vP_n}
\def\vPHn{\vPn^{\sv}}
\def\vPhn{\vPn^{\cc}}

\def\F{\textrm{F}}
\def\Fn{{\F_n}}
\def\RF{\textrm{RF}}
\def\RFn{{\RF_n}}
\def\Aut{\textrm{Aut}}
\def\AutC{\Aut_{\textrm{C}}}

\def\DDsize{.9cm}
\def\DDAa{\dessin{\DDsize}{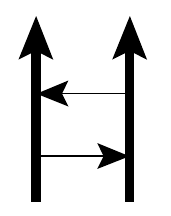}}
\def\DDAb{\dessin{\DDsize}{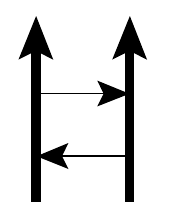}}
\def\DDBa{\dessin{\DDsize}{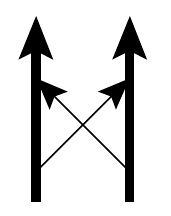}}
\def\DDCa{\dessin{\DDsize}{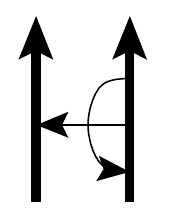}}
\def\DDDa{\dessin{\DDsize}{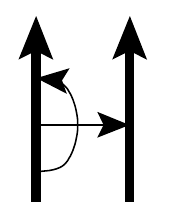}}
\def\DDCb{\dessin{\DDsize}{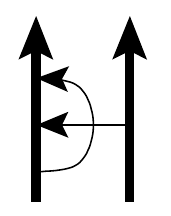}}
\def\DDDb{\dessin{\DDsize}{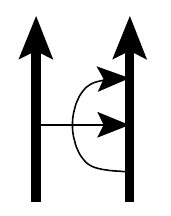}}
\def\DDBc{\dessin{\DDsize}{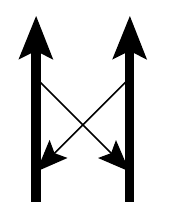}}
\def\DDCc{\dessin{\DDsize}{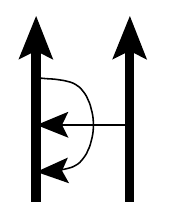}}
\def\DDDc{\dessin{\DDsize}{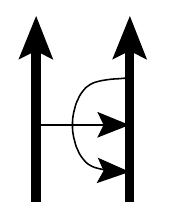}}
\def\DDCd{\dessin{\DDsize}{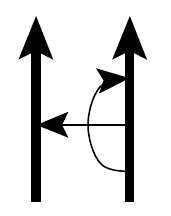}}
\def\DDDf{\dessin{\DDsize}{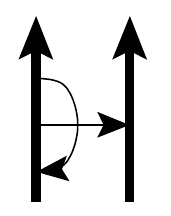}}
\def\DDEa{\dessin{\DDsize}{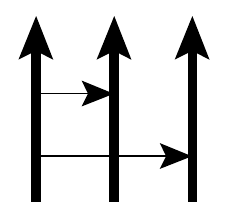}}
\def\DDEb{\dessin{\DDsize}{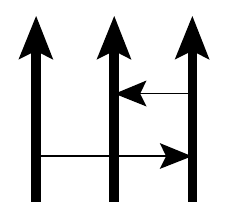}}
\def\DDEc{\dessin{\DDsize}{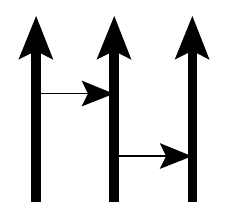}}
\def\VTO{\dessin{\DDsize}{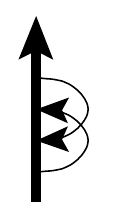}}
\def\VT{\dessin{\DDsize}{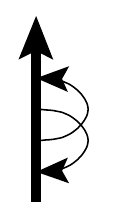}}
\def\VTU{\dessin{\DDsize}{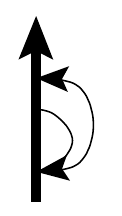}}
\def\VTD{\dessin{\DDsize}{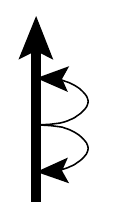}}
\def\VTT{\dessin{\DDsize}{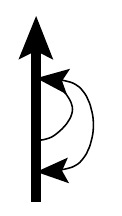}}

\begin{document}

\title[Usual, Virtual and Welded knotted objects up to homotopy]{On Usual, Virtual and Welded knotted objects up to homotopy}
\author[B. Audoux]{Benjamin Audoux}
         \address{Aix Marseille Universit\'e, I2M, UMR 7373, 13453 Marseille, France}
         \email{benjamin.audoux@univ-amu.fr}
\author[P. Bellingeri]{Paolo Bellingeri}
         \address{Universit\'e de Caen, LMNO, 14032 Caen, France}
         \email{paolo.bellingeri@unicaen.fr}
\author[J.B. Meilhan]{Jean-Baptiste Meilhan}
         \address{Universit\'e Grenoble Alpes, IF, 38000 Grenoble, France}
         \email{jean-baptiste.meilhan@ujf-grenoble.fr}
\author[E. Wagner]{Emmanuel Wagner}
         \address{IMB UMR5584, CNRS, Univ. Bourgogne Franche-Comt\'e, F-21000 Dijon, France}
         \email{emmanuel.wagner@u-bourgogne.fr}



%\recdate{}{}{}% Date of reception
%\revdate{}{}{}%  Date of revision

\begin{abstract}
We consider several classes of knotted objects, namely usual, virtual and welded pure braids and string links, 
and two equivalence relations on those objects, induced by either self-crossing changes or self-virtualizations. 
We provide a number of results which point out the differences between these various notions. 
The proofs are mainly based on the techniques of Gauss diagram formulae.
\end{abstract}

\maketitle

\section{Introduction}

In this note, we study several variations of the notions of pure braids and string links. 
Recall that string links are pure tangles without closed components, which form a monoid that contains the pure braid group as the group of units. 
As usual in knot theory, these objects can be regarded as diagrams up to Reidemeister moves. 
When allowing virtual crossings in such diagrams, modulo a suitably extended set of Reidemeister moves, one defines 
the notions of virtual pure braids and virtual string links. 
Another related class of object is that of welded knotted objects.
Welded knots are a natural quotient of virtual knots, by the so-called Overcrossings Commute relation, which is one of the two forbidden moves in virtual knot theory.
What makes this Overcrossings Commute relation natural is that the virtual knot group, and hence any virtual knot invariant derived from it, factors through it.
These welded knotted objects first appeared in a work of Fenn-Rimanyi-Rourke in the more algebraic context of braids \cite{FRR}.
The study of these three classes (usual, virtual and welded) of knotted objects is currently the subject of an ongoing project of Bar-Natan and Dancso \cite{WKO1,WKO2,Hoo},
which aims at relating certain algebraic structures to the finite type theories for these objects.

Works of Habegger and Lin \cite{HL} show that, in the usual case, any string link is link-homotopic to a pure braid, 
and that string links are completely classified up to link-homotopy by their action on the reduced free group $\RFn$. 
Here, the link-homotopy is the equivalence relation on knotted object generated by self-crossing changes, and
the reduced free group is the smallest quotient of the free group $\Fn$ where each generator commutes with all of its conjugates.  \\
In \cite{vaskho}, the authors gave welded analogues of these results, recalled in Theorems \ref{th:wSLh=wPh} and \ref{th:wSLh=AutC} below. 
There, it appears that the right analogue of link-homotopy in the virtual/welded setting is the notion of $\sv$-equivalence, which is the 
equivalence relation on virtual knotted objects generated by self-virtualization, i.e. replacement of a classical self crossing by a virtual one. 

This note contains a series of results which analyse further the various quotients of 
usual, virtual and welded pure braids and string links up to link-homotopy and $\sv$-equivalence, and the relations between them. 
The summary of our results, stated and proved in Section \ref{sec:miscellaneous}, is given in Figure \ref{fig:Connections} below. 
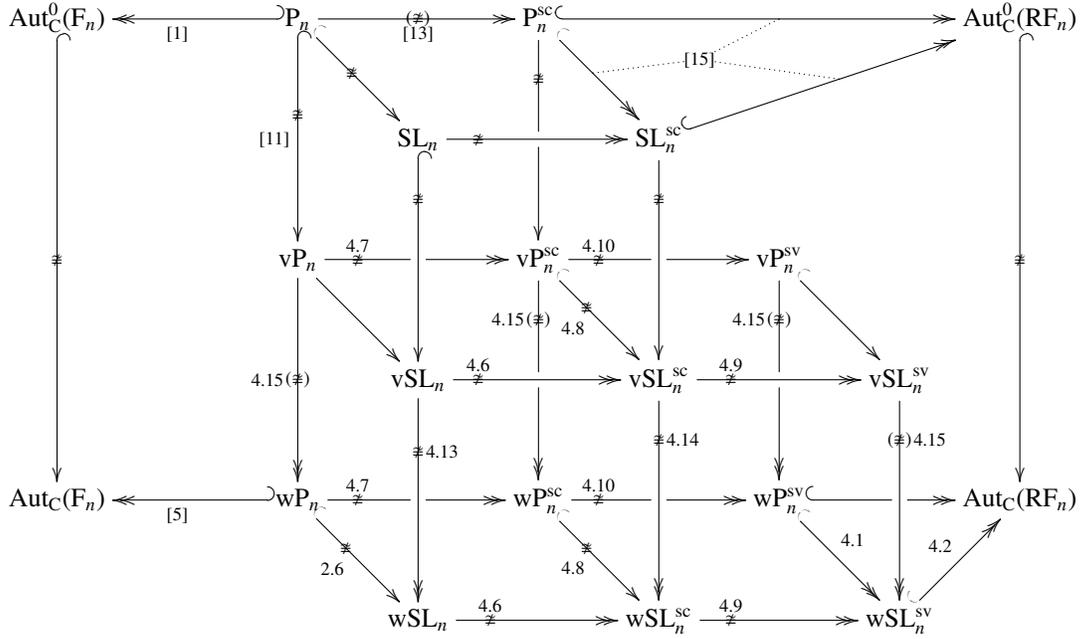
\begin{figure}[h]
\[
\vcenter{\hbox{\xymatrix@!0@C=1.6cm@R=1.6cm{
\AutC^0(\Fn)\ar@{<<-^{)}}[rr]_{\mbox{\scriptsize\cite{artin}}} \ar@{^{(}->}[dddd]\ar@{}[dddd]|{\ncong}&&
\Pn \ar@{->>}[rr]\ar@{}[rr]|{(\ncong)}_{\mbox{\scriptsize\cite{Gold}}} \ar@{^{(}->}[rd]\ar@{}[rd]|(.45){\ncong} \ar@{^{(}->}[dd] \ar@{}[dd]|(.4){\ncong}_{\mbox{\scriptsize\cite{FRR}}}
&&
\Phn \ar@{^{(}->>}[rrrr]|{}="M2" \ar@{^{(}->>}[rd]|(.45){}="M1" \ar[dd]|!{[dl];[dr]}\hole \ar@{}[d]|{\ncong}&&&&
\AutC^0(\RFn)\ar@{}[]|{}="FIN" \ar@{} "M1";"FIN"|(.25){\mbox{\scriptsize\cite{HL}}}="CENTRE"\\
&&& \SLn \ar@{->>}[rr]\ar@{}[r]|{\ncong} \ar@{^{(}->}[dd] \ar@{}[d]|{\ncong}&&
\SLhn \ar@{^{(}->>}[urrr]|{}="M3" \ar[dd] \ar@{}[d]|{\ncong}&&& \ar@{.}
"CENTRE";"M1" \ar@{.} "CENTRE";"M2" \ar@{.} "CENTRE";"M3"\\
&& \vPn
\ar@{->>}[rr]|!{[ur];[dr]}\hole \ar@{}[r]|{\ncong}^{\ref{prop:wPh}} \ar@{->}[dr]\ar@{}[dr]
\ar@{}[]|{}="AAa" \ar@{->>}[dd]
\ar@{}[dd]|{(\ncong)}_{\ref{prop:WeldedGenuine}\,\,} &&
\vPhn \ar@{->>}[rr]|!{[ur];[dr]}\hole \ar@{}[r]|{\ncong}^{\ref{purebraids}}
\ar@{->>}[dd]|!{[dl];[dr]}\hole \ar@{}[d]|{(\ncong)}_{\ref{prop:WeldedGenuine}\,\,} \ar@{^{(}->}[rd] \ar@{}[dr]|(.4){\ncong}_(.45){\ref{prop:PSLh}} 
&&
\vPHn \ar@{->>}[dd]|!{[dl];[dr]}\hole \ar@{}[d]|{(\ncong)}_{\ref{prop:WeldedGenuine} \,\,}\ar@{^{(}->}[rd] && \\
&&& \vSLn \ar@{->>}[rr] \ar@{->>}[dd] \ar@{}[d]|(.6){\ncong}^(.6){\ref{lem:vwSL}}
\ar@{}[r]|{\ncong}^{\ref{prop:wSLh}}
&&
\vSLhn \ar@{->>}[rr]\ar@{}[r]|(.6){\ncong}^(.6){\ref{prop:wSLhH}} \ar@{->>}[dd]\ar@{}[d]|{\ncong}^{\ref{rem:knottedness}}
&&
\vSLHn \ar@{->>}[dd] \ar@{}[d]|{(\ncong)}^{\,\, \ref{prop:WeldedGenuine} }& \\
\AutC(\Fn) \ar@{<<-^{)}}[rr]_{\mbox{\scriptsize\cite{WKO1}}}&&\wPn
\ar@{->>}[rr]|!{[ur];[dr]}\hole \ar@{}[r]|{\ncong}^{\ref{prop:wPh}} \ar@{^{(}->}[dr]\ar@{}[dr]|(.4){\ncong}_(.45){\ref{rem:wTBW}} &&
\wPhn \ar@{->>}[rr]|!{[ur];[dr]}\hole \ar@{}[r]|{\ncong}^{\ref{purebraids}} \ar@{^{(}->}[dr]
\ar@{}[dr]|(.4){\ncong}_(.45){\ref{prop:PSLh}} &&
\wPHn \ar@{^{(}->>}[rr]|!{[ur];[dr]}\hole \ar@{^{(}->>}[dr]\ar@{}[dr]^(.45){\ref{th:wSLh=wPh}} &&
\AutC(\RFn) \ar@{<-_{)}}[uuuu] \ar@{}[uuuu]|{\ncong}\\
&&& \wSLn \ar@{->>}[rr] \ar@{}[r]|(.6){\ncong}^(.6){\ref{prop:wSLh}}&&
\wSLhn \ar@{->>}[rr] \ar@{}[r]|(.6){\ncong}^(.6){\ref{prop:wSLhH}} &&
\wSLHn \ar@{^{(}->>}[ur]\ar@{}[ur]^{\ref{th:wSLh=AutC}} &
}}}
\]
\caption{Connections between usual/virtual/welded pure braids and string links.\\
  {\footnotesize All statements hold for $n\geq2$ except for the 
    ``$(\ncong)$", which become isomorphisms for $n=2$}}
\label{fig:Connections}
\end{figure}
Although all notation and definitions needed for this diagram will be given in Section \ref{sec:defi}, let us outline here that  
\begin{itemize}
 \item $\Pn$ and $\SLn$ stand for the (usual) sets of pure braids and string links on $n$ strands, 
       and the prefix v and w refer to their virtual and welded counterpart, respectively; 
 \item the superscripts $\sv$ and $\cc$ refer respectively to the equivalence relations generated by self-virtualization and self-crossing change. 
\end{itemize}

\begin{acknowledgments}
We wish to thank Dror Bar-Natan and Arnaud Mortier for numerous fruitful discussions. 
We also thank Anne Isabel Gaudreau and Akira Yasuhara for their comments on the first version of this paper. 
Thanks are also due to the anonymous referee for his/her careful reading of the manuscript.  
This work is supported by the French ANR research project ``VasKho'' ANR-11-JS01-00201. 
\end{acknowledgments}

%%%%%%%%%%%%%%%%%%%%%%%%%%%%%%%%%%%%%%%%%%%%%%%%%%%%%%%%%%%%%%%%%%%%%%%%%%%
\section{Definitions}\label{sec:defi} 

Before we define below the main objects of this note, let us fix a few notation that will be used throughout. % the paper. 

We set $n$ to be a non negative integer, and we denote by $\UnN$ the set of integers between $1$ and $n$. 
Denote by $I$ the unit closed interval. We fix $n$ distinct points $\{p_i\}_{i\in\UnN}$ in $I$.   
We will also use the following algebraic notation.
Let $\Fn$ be the free group on $n$ generators $x_1,\ldots,x_n$.
We denote by $\RFn$ the quotient group of $\Fn$ by the normal closure of the subgroup generated by elements of the form $[x_i,g^{-1} x_i g]$, where $i\in \{1,\cdots,n\}$ and $g\in \Fn$;  
this quotient is called the \emph{reduced free group} on $n$ generators. 

Let $G$ be either $\Fn$ or $\RFn$ with its chosen system of generators; we define
   \begin{itemize}
  \item $\AutC(G):=\big\{f\in\Aut(G)\ \big|\ \forall
    i\in\{1,\cdots,n\},\exists g\in G,f(x_i)=g^{-1} x_i g\big\}$, the group of \emph{basis-conjugating automorphisms} of $G$; 
  \item $\AutC^0(G):=\big\{f\in\AutC(G)\ \big|\ f(x_1\cdots x_n)=x_1\cdots x_n\big\}$.
  \end{itemize}

\subsection{Usual, virtual and welded knotted objects}

In this section, we introduce the main objects of this note. 

\begin{defi}\label{def:StringLink}
An $n$-component \emph{virtual string link diagram} is an immersion $L$ of $n$ oriented intervals $\psqcup_{i=1}^n I_i$ in $I\times I$, called \emph{strands}, such that:
\begin{itemize}
\item each strand $I_i$ has boundary $\p I_i=\{p_i\}\times\{0,1\}$ and  is oriented from $\{p_i\}\times\{0\}$ to $\{p_i\}\times\{1\}$ ($i\in\UnN$);
\item the singular set of $L$ is a finite set of transverse double points.
\item a decoration is added at each double point, and the decorated double point is called either \emph{a classical crossing} or \emph{a virtual crossing}, as indicated in Figure \ref{fig:crossings}. 
\end{itemize}
A classical crossing where the two preimages belong to the same component is called a \emph{self-crossing}.   
\end{defi}

Up to isotopy, the set of virtual string link diagrams is naturally endowed with a monoidal structure by the stacking product, 
and with unit element the trivial diagram $\pcup_{i=1}^n p_i\times I$. 

\begin{figure}[!h]
\[
\begin{array}{ccc}
  \dessin{2cm}{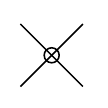}&\hspace{1cm}&\vcenter{\hbox{\xymatrix@R=-1cm{&\dessin{1.6cm}{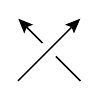}\
        \textrm{\ : positive}\\\dessin{2cm}{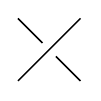}\ar[ur]\ar[dr]&\\&\dessin{1.6cm}{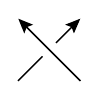}\
        \textrm{\ : negative}}}}\\
  \textrm{virtual}&&\textrm{classical}\hspace{1.5cm}
\end{array}
\]
 \caption{Virtual and  classical crossings} \label{fig:crossings}
\end{figure}

Two virtual string link diagrams are equivalent if they are related by a finite sequence of  \emph{generalized Reidemeister moves}, represented in  Figure \ref{fig:AllReidMoves}. 
As is well known, virtual and mixed Reidemeister moves imply the more general \emph{detour move}, which replaces an arc going through only virtual crossings by any other such arc, fixing the boundary \cite{Kauffman}.   
  
\begin{figure}[!h]
  \[
  \begin{array}{c}
    \begin{array}{c}
       \textrm{RI}:\ \dessin{1.5cm}{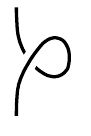}
    \leftrightarrow \dessin{1.5cm}{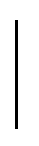} \leftrightarrow \dessin{1.5cm}{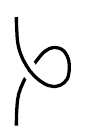}
\hspace{.7cm}
\textrm{RII}:\ \dessin{1.5cm}{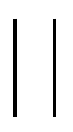} \leftrightarrow \dessin{1.5cm}{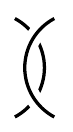}\\[1cm]
\textrm{RIIIa}:\ \dessin{1.5cm}{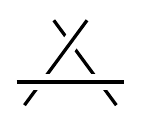} \leftrightarrow \dessin{1.5cm}{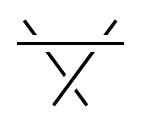}
\hspace{.7cm}
\textrm{RIIIb}:\ \dessin{1.5cm}{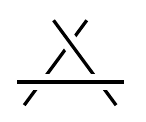} \leftrightarrow \dessin{1.5cm}{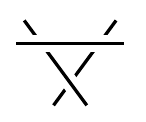}
\end{array}\\[2cm]
    \textrm{classical  Reidemeister moves}
  \end{array}
  \]
\vspace{.5cm}
  \[
     \begin{array}{c}
       \textrm{vRI}:\ \dessin{1.5cm}{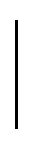}
    \leftrightarrow \dessin{1.5cm}{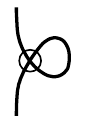}
\hspace{.7cm}
\textrm{vRII}:\ \dessin{1.5cm}{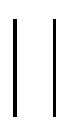} \leftrightarrow \dessin{1.5cm}{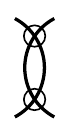}
\hspace{.7cm}
\textrm{vRIII}:\ \dessin{1.5cm}{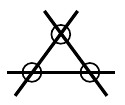} \leftrightarrow \dessin{1.5cm}{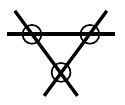}\\[1cm]
    \textrm{virtual Reidemeister  moves}
  \end{array}
   \]
\vspace{.5cm}
  \[
      \begin{array}{c}
       \textrm{mRIIIa}:\ \dessin{1.5cm}{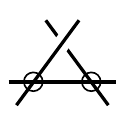}
    \leftrightarrow \dessin{1.5cm}{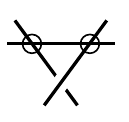}
\hspace{.7cm}
\textrm{mRIIIb}:\ \dessin{1.5cm}{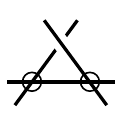} \leftrightarrow \dessin{1.5cm}{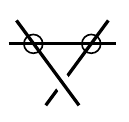}\\[1cm]
    \textrm{mixed Reidemeister  moves}
  \end{array}
  \]
  \caption{Generalized Reidemeister moves on diagrams; here, all lines are pieces of strands which may belong to the same strand or not, and can have any orientation.}  \label{fig:AllReidMoves}
\end{figure}

We denote by $\vSLn$ the quotient of $n$-component virtual string link diagrams up to isotopy and generalized Reidemeister moves, which is a monoid with composition induced by the stacking product. We call its elements $n$-component \emph{virtual string links}.
\begin{remarque}
In the literature, $1$-component (virtual) string links were previously
studied under the name of (virtual) long knots. We shall sometimes make
use of this terminology.
\end{remarque}
\begin{defi}\label{def:MonotoneSLn}
An  $n$-component virtual string link diagram is \emph{monotone} if it intersects $I\times \{t\}$ at $n$ points for all $t\in I$, where double points are counted with multiplicity.  
\end{defi}
We denote by  $\vPn$ the monoid of monotone elements of $\vSLn$, up to monotone transformations.  This monoid is in fact the group of virtual pure braids studied in \cite{Bardakov}.   
\begin{remarque}\label{rem:VirtualBraids}
One could also consider $\vPn$ to be the set of diagrams admitting a monotone representative (up to non-monotone transformations). 
The question of whether these to definitions agree is equivalent to that of the embedding of $\vPn$ into $\vSLn$; see Question \ref{q:q}.  
\end{remarque}

As explained in the introduction, there is a natural quotient of virtual knot theory, where one of the forbidden moves is allowed: 
\begin{defi}
We define the Overcrossings Commute (OC) move as
\[
\OC:\ \dessin{1.5cm}{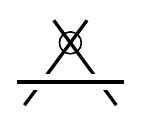}\ \longleftrightarrow \dessin{1.5cm}{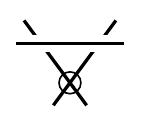}.
\]
\noi We denote by $\wSLn:=\fract{\vSLn}/{\OC}$ the quotient of $\vSLn$
  up to $\OC$ moves, which inherits a monoid structure from the
  stacking product. We call its elements
  $n$-component \emph{welded string links}.
\end{defi}
  We denote by  $\wPn$ the submonoid of $\wSLn$ of monotone elements up to monotone transformations. 
This monoid is in fact the welded pure braid group studied, for instance, in \cite{WKO1}. Thus, we will freely call welded pure braids the elements of $\wPn$. 
\begin{remarque}\label{rem:wTBW}
Unlike in the virtual case, the welded pure braid group is known to be isomorphic to the subset of $\wSLn$ admitting a monotone representative, see \cite[Rem.~3.7]{vaskho}. 
\end{remarque}

\begin{warnings}$\ $
\begin{itemize}
\item The following  Undercrossings Commute (UC) move
    \[
    \begin{array}{rcl}
      \UC:\ \dessin{1.5cm}{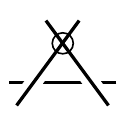}&\ \longleftrightarrow\
      &\dessin{1.5cm}{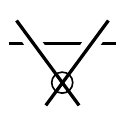}\\[-.96cm]
      &\textrm{\textcolor{red}{$\times$}}&\\[.6cm]
    \end{array},
    \]
    \noi was forbidden in the virtual context and is still forbidden in the welded context.
\item Virtual and welded notions do not coincide, even for $n=1$, where we get respectively the notion of virtual and welded long knots (see \cite{WKO1}, and also Section \ref{sec:miscellaneous}, for a summary of results).
\end{itemize}
\end{warnings}

Let us now turn to the classical versions of these objects. 
\begin{defi}\label{def:uStringLink}
An $n$-component \emph{string link} is an embedding $L=\psqcup_{i=1}^n I_i$ of $n$ disjoint copies of the oriented interval $I$ in the the standard cube  $I^3$ 
such that $I_i$ runs from $\left(p_i,\frac{1}{2},0\right)$ to $\left(p_i,\frac{1}{2},1\right)$ for all $i\in\UnN$. 
We denote by $\SLn$ the set of $n$-component classical string link up to isotopy. 
It is naturally endowed with a monoidal structure by the stacking product, with unit $\cup_{i=1}^n \{\left(p_i,\frac{1}{2}\right)\} \times I$.  
Borrowing the terminology of \cite{WKO1,WKO2}, we shall call such string links \emph{usual} in order to distinguish them from the virtual and welded ones. 
\end{defi}

Any usual string link can be generically projected in $I^2$ onto a
virtual string link diagram with classical crossings only. Hence
$\SLn$ can be described as the set of virtual string link diagrams
with no virtual crossing, modulo the classical Reidemeister moves
only.   
Actually, classical string links embed in their virtual or welded counterpart. 
This is shown strictly as in the knot case \cite[Thm.1.B]{GPV} ; in particular, this uses the fact that the group system is a complete invariant, and more specifically the fact that string link complements are Haken manifolds.  

We also denote by $\Pn$ the (usual) pure braid group on $n$ strands,
which can likewise be seen as the set of monotone virtual string links
with no virtual crossing.  
Recall from  \cite{FRR} that this map is known to be a well-defined embedding  of $\Pn$  in $\vPn$ and in $\wPn$. 

\begin{remarque}\label{rem:topo}
Unlike the class of usual knotted objects, which is intrinsically topological, virtual and welded objects are diagrammatical in nature. 
However, both theories enjoy nice topological interpretations. \\
It is now well known and understood  \cite{CKS,Kuperberg} that virtual knot theory can be realized topologically as the theory of knots in thickened surfaces modulo handle stabilization.  Note that the Overcrossings Commute relation is not satisfied in this topological setting. \\
A topological realization of welded theory is given by considering a certain class of surfaces embedded in $4$-space. In particular, welded string links map surjectively onto the monoid of ribbon tubes studied in \cite{vaskho}.  This had first been pointed out by Satoh for the case of welded knots \cite{Citare}, but some key ideas already appeared in early works of Yajima \cite{yajima}. 
\end{remarque}

\subsection{Self-local moves and homotopy relations}\label{sec:homo}

In this note, we consider two types of equivalence relations on the above usual/virtual/welded objects, both generated by self-local moves.

Two virtual string link diagrams are related by a \emph{self-virtualization} if one can be obtained from the other by turning a classical self-crossing into a virtual one.
\begin{defi}\label{def:selfvirtualization}
The \emph{$\sv$--equivalence} is the equivalence relation on virtual knotted objects generated by self-virtualizations. 
\end{defi}
  
We denote respectively by $\vSLHn$ and $\wSLHn$ the quotient of $\vSLn$ and $\wSLn$ under $\sv$--equivalence, which are monoids with composition induced by the stacking product.
We also denote by $\vPHn\subset\vSLHn$ and $\wPHn\subset\wSLHn$ the respective subsets of elements having a monotone representative.

There is also a natural notion of crossing change, which is a local move that switches a positive classical crossing to a negative one, and vice-versa.  If we further require that the two strands involved belong to the same component, we define a \emph{self-crossing change}. 

The classical notion of link-homotopy is the equivalence relation on usual (string) links generated by self-crossing changes.  It was introduced for links by Milnor in \cite{Milnor}, and later used by Habegger and Lin for string links \cite{HL},  in order to ``study (string) links modulo knot theory'',  and focus on the interactions between distinct components. 
More precisely, link-homotopy on usual string links allows not only to unknot each component individually,  but also simultaneously, since every usual string link is link-homotopic to a pure braid.

Thanks to the local nature of crossing changes, the notion of link-homotopy can be extended to the whole monoids $\vSLn$ and $\wSLn$. 
\begin{defi}
The \emph{$\cc$-equivalence} is the equivalence relation on virtual knotted objects generated by self-crossing changes. 
\end{defi}
We denote the quotient of $\vSLn$ under $\cc$-equivalence by $\vSLhn$. In addition, we denote by $\vPhn$ the image of $\vPn$ in $\vSLhn$.  We shall use similar notation in the usual and welded cases. 
\begin{remarque}\label{rem:flat}
 Note that, in the \emph{one-component} case, virtual knotted object up to $\cc$-equivalence coincide with the notion of flat virtual knotted objects introduced in \cite{Kauffman}.  
\end{remarque}

Since a crossing change can be realized by a sequence of two (de)virtualization moves, the $\cc$-equivalence is clearly sharper than the $\sv$--equivalence. 
It is also {\it a priori} a more natural extension of the classical situation. However, as already noted in \cite{vaskho}, 
it appears not to be the relevant notion for the study of welded string links ``modulo knot theory'';  
this is recalled in further details in Section \ref{sec:citare}.  

%%%%%%%%%%%%%%%%%%%%%%%%%%%%%%%%%%%%%%%%%%%%%%%%%%%%%%%%%%%%%%%%%%%%%%%%%%

\section{Gauss diagram formulae for virtual and welded string links}\label{sec:GD}

In this section, we recall the main tools for proving the results of this paper, namely Gauss diagram formulae \cite{F,GPV,PV}. 
\subsection{Gauss diagrams} 
We first roughly review the notion of Gauss diagrams. 
\begin{defi}
A \emph{Gauss diagram} $G$ is a set of signed and oriented (thin) arrows
between points of $n$ ordered and oriented vertical (thick) strands,
up to isotopy of the underlying strands. Endpoints of arrows are
called \emph{ends} and are divided in two parts, \emph{heads} and
\emph{tails},  defined by the orientation of the arrow (which goes by
convention from the tail to the head). An arrow having both ends on
the same strand is called a \emph{self-arrow}.
Any Gauss diagram obtained by the deletion of several
(possibly none) arrows of $G$ is called a subdiagram of $G$.
\end{defi}
Examples of Gauss diagrams can be found in Figures \ref{fig:Kishino}
to \ref{fig:pure}, \ref{fig:sl2bis} and \ref{fig:CC'} in the next section.  
As these figures also illustrate, Gauss diagrams serve as a combinatorial tool for faithfully encoding virtual/welded knotted objects.  
Indeed, it is well known that for any virtual string link diagram $L$, there is a unique associated Gauss diagram $G_L$, 
where the set of classical crossings in $L$ is in one-to-one correspondence with  the set of arrows in $G_L$, 
and that this correspondence induces a bijection between $\vSL_n$ and the set of Gauss diagrams up to the natural analogues 
of classical Reidemeister moves\footnote{Note  that there are no Gauss diagram analogues of the mixed and virtual Reidemeister moves, 
since virtual crossings are simply not materialized in Gauss diagrams. }  of Figure \ref{fig:AllReidMoves}. See \cite{F,GPV} for the usual link case. 

Likewise, welded diagrams are faithfully encoded by equivalence classes of Gauss diagram up to the following Tails Commute (TC) move, which is the Gauss diagram analogue of Overcrossings Commute: 
\[
\TC:\ \dessin{2cm}{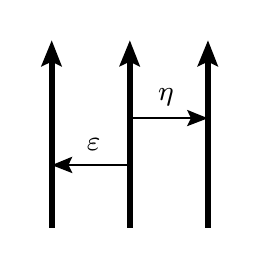}\ \longleftrightarrow\ \dessin{2cm}{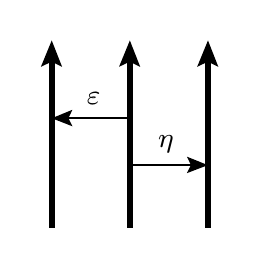},
\]
where the signs $\varepsilon$ and $\eta$ are arbitrary.

Finally, the next observation allows to study the two homotopy relations introduced in Section \ref{sec:homo}. 
At the level of Gauss diagrams, the $\cc$-equivalence is generated by the local move which switches both the orientation and the sign of a
self-arrow; the $\sv$-equivalence, on the other hand, is simply generated by the removal of a self-arrow. 

\subsection{Gauss diagrams formulae} 
We now review Gauss diagrams formulae. 
First, let us define an \emph{arrow diagram} to be  an unsigned Gauss diagram
i.e. an arrow diagram on $n$ strands of $n$ intervals with unsigned arrows (see \cite{polyak_arrow}).

Given a Gauss diagram $G$, there is an associated formal linear combination of arrow diagrams 
 \[
i(G) := \sum_{G'\subseteq G} \sigma(G')A_ {G'},
\]
where the sum runs over all subdiagrams of $G$, $\sigma(G')$ denotes the product of the signs  in the subdiagram $G'$ and $A_{G'}$ is the arrow diagram obtained from $G'$ by forgetting the signs. 

The $\mathbb{Z}$-module $\AA_n$ generated by arrow diagrams on $n$ strands comes equipped with a natural scalar product $(-,-)$, 
defined by $(A,A')=\delta_{A,A'}$ for any two arrow diagrams $A$ and $A'$, where $\delta$ denotes the Kronecker delta symbol.  
So given any formal linear combination $F$ of arrow diagrams in $\AA_n$, one can define a map on the set of Gauss diagrams on $n$ strands $\langle F ; - \rangle$ by setting 
 $$ \langle F ; G \rangle : =  \big( F,i(G) \big) $$ 
 for any Gauss diagram $G$.  Roughly speaking, this map counts with signs subdiagrams of $G$. 

We can define in this way a map on the set of virtual diagrams by setting $\langle F ; L \rangle := \langle F ; G_L \rangle$,  for any virtual string link diagram $L$ with associated Gauss diagram $G_L$.   We say that $F$ is the \emph{defining linear combination} of the map $\langle F ; - \rangle$. 
Now, such a map does not, in general, factor through the generalized Reidemeister moves.  We recall below a simple criterion, due to Mortier \cite{mortier}, which gives a sufficient condition for getting a virtual string link invariant in this way.  To state this criterion, we need a few more definitions.

A \emph{degenerate arrow diagram} is an arrow diagram where two arrow ends are allowed to coincide. We denote by $\DD_n$ the abelian group freely generated by degenerate arrow diagrams on $n$ strands, modulo the relations: 
\[
\dessin{1.2cm}{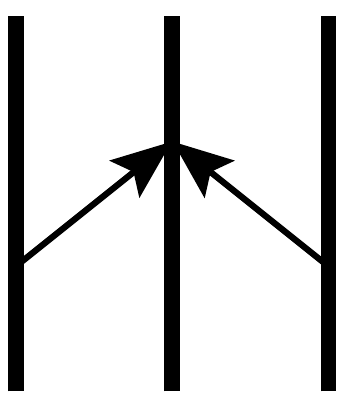}\ =\ \dessin{1.2cm}{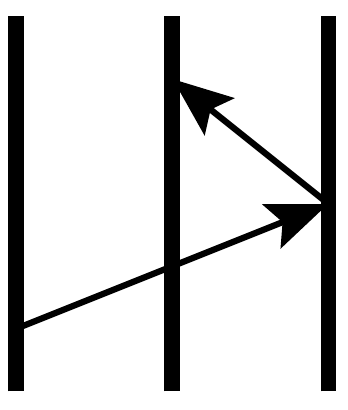}\ +\ \dessin{1.2cm}{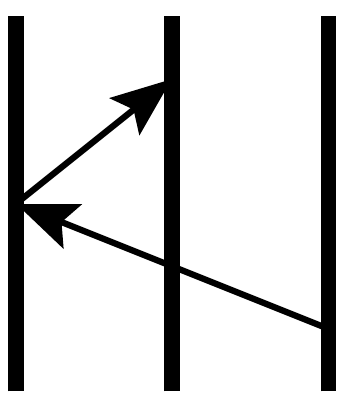}\quad ;\quad \dessin{1.2cm}{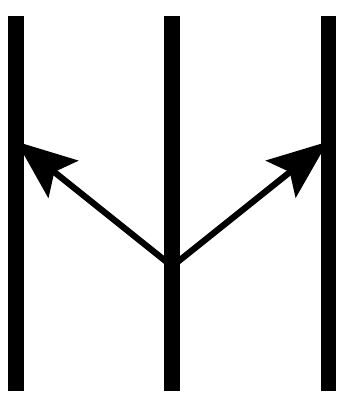}\ =\ \dessin{1.2cm}{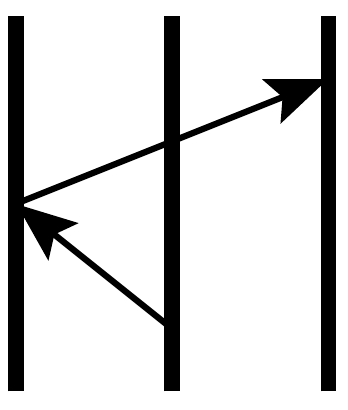}\ +\ \dessin{1.2cm}{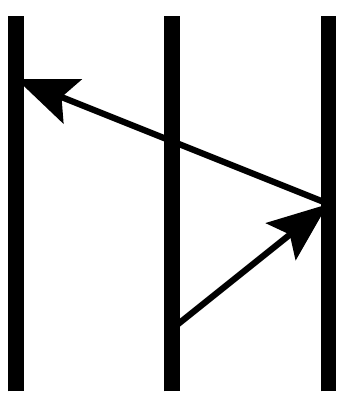}  . 
\]
In each term of the above relations, the three vertical lines are portions of the $n$ vertical strands of a degenerate arrow diagram, which may belong to the same component or not. 

Two arrow ends are called \emph{adjacent} if they are met consecutively when running along some strand.   An \emph{internal edge} of an arrow diagram is a portion of a strand cobounded by two \emph{distinct} adjacent arrow ends.   We define a linear map $d: \AA_n\rightarrow \DD_n$ by sending any arrow diagram $A$ on $n$ strands to 
\[
d(A) : = \sum_{\textrm{internal edges $e$ of $A$}}
(-1)^{\uparrow_e}.\eta(e).A_e,
\]
\noi where $A_e\in \DD_n$ denote the degenerate arrow diagram obtained by shrinking $e$ to a point,   $\uparrow_e \in \{ 0,1,2 \}$ is the number of arrow heads bounding $e$,  and $\eta(e)\in \{ \pm 1 \}$ is given by 
\[
\eta(e)=\left\{ \begin{array}{cl}
       -1 & \textrm{if the two arrows cobounding $e$ do not cross in $A$,} \\
       +1 & \textrm{otherwise,}
       \end{array}\right.
\]
\noi with the convention that two arrows \emph{do not cross} in an arrow diagram if, when running along the $n$ strands $I_1$ to $I_n$, in order and following the orientations, we meet the two ends of one of these arrows consecutively. 

\begin{theo}\label{thm:mortier}
Let $F\in \AA_n$ be a linear combination of arrow diagrams on $n$ strands:
\begin{enumerate}
\item if $F$ does not contain any arrow with adjacent ends, then  $\langle F ; - \rangle$ is invariant under move RI;
\item if $F$ does not contain two arrows with adjacent heads and adjacent tails, then  $\langle F ; - \rangle$ is invariant under move RII;
\item if $\langle F ; - \rangle$ is invariant under move RII and if $d(F)$ is zero in $\DD_n$, then  $\langle F ; - \rangle$ is invariant under moves RIII;
\item if $F$ does not contain any pair of adjacent arrow tails, then  $\langle F ; - \rangle$ is invariant under $\OC$;
\item if, for each diagram of $F$ which has a self-arrow $\vec a$, the diagram obtained by reversing the orientation of $\vec a$ also appears in $F$ with opposite sign, then  $\langle F ; - \rangle$ is invariant under self-crossing change;
\item if $F$ does not contain any self-arrow, then $\langle F ; - \rangle$ is invariant under self-virtualization. 
\end{enumerate}
\end{theo}

The only difficult part of the statement is the invariance under RIII and it is due to A. Mortier.  It can be found in \cite{mortier} where it is stated as an equivalence in the context of virtual knots.  However, the arguments adapt verbatim to the case of virtual string links. 
Actually, Mortier pointed out the fact that each point in Theorem \ref{thm:mortier} is also an equivalence in the string link case (this fact can be proved using a suitable Polyak algebra).

\begin{exemple}
As an example, let us consider the two invariants $v_{2,1}$ and $v_{2,2}$, defined in \cite{GPV} by 
\[
v_{2,1} = \left\langle \VTO , - \right\rangle 
\quad 
\textrm{ and }
\quad
v_{2,2} = \left\langle \VT , - \right\rangle.
\]
\noi They are easily seen to be invariants of virtual $1$-string links.  Indeed, for, say, the latter one, we have 
\[ d\left( \VT \right) = -\VTU + \VTD - \VTT = 0\in \mathcal{D}_1. \]
The invariant $v_{2,1}$ is moreover a welded $1$-string links invariant, while $v_{2,2}$ is not, 
since the defining diagram of the latter contains two adjacent arrow tails. 
As a matter of fact, the virtual string link $K$ of Figure \ref{fig:sl1} is trivial in $\wSL_1$, but we have $v_{2,2}(K)=-1$.
\end{exemple}

In the rest of the paper, we leave it as an exercise to the reader to check using Theorem \ref{thm:mortier} 
that each invariant defined via an arrow diagram formula has the desired invariance properties. 

\section{Results on usual, virtual and welded braid-like objects} \label{sec:miscellaneous}

In this section, we recall some comparative results on 
usual, virtual and welded knotted objects, 
and we provide further results comparing various notions of homotopy for these objects.
They are roughly summarized in Figure \ref{fig:Connections}.

\subsection{Some analogies between the usual and welded theories} \label{sec:citare}
Let us start by recalling a couple of results from  \cite{vaskho}, on the $\sv$--equivalence for welded string links.  
On one hand, we have the following 
\begin{theo}[\cite{vaskho}]  \label{th:wSLh=wPh}
Every welded string link is monotone up to self-virtualization. 
\end{theo}
It follows in particular that $\wSLHn$ is a group. 
Theorem \ref{th:wSLh=wPh} is a welded analogue of a result of Habegger an Lin \cite{HL}, which states that any usual string link is link-homotopic to a pure braid.

On the other hand, we have a classification result, analogous to \cite[Thm.~1.7]{HL} in the usual case.   
\begin{theo}[\cite{vaskho}] \label{th:wSLh=AutC}
The groups $\wSLHn$, $\wPHn$ and $\AutC(\RFn)$ are all isomorphic. 
\end{theo}
\noindent Note, moreover, that this classification of welded string links up to $\sv$--equivalence is achieved by a virtual extension of Milnor invariants. 

The next theorem illustrates the fact, suggested by the above results, that the $\sv$--equivalence can indeed be seen as a natural extension of the usual link-homotopy to the welded case. 
\begin{theo}\label{homotopyequiv}
Let $L_1, L_2 \in \SLn$, and let $\iota_w:  \SLn \to \wSLn$ be the natural map induced by the inclusion at the level of diagrams.
If  $\iota_w(L_1)$ and $\iota_w(L_2)$ are $\sv$--equivalent, then $L_1$ and $L_2$ are $\cc$-equivalent. 
\end{theo}
\noindent In other words, the notion of self-virtualization restricted to welded string links with only classical crossings coincides with the usual self-crossing change. 
\begin{proof}
Recall that $\AutC^0(\RF_n)$ is the set of automorphisms in $\AutC(\RF_n)$ which leave the product $x_1 x_2 \cdots x_n$ invariant.  
Theorem 1.7 of \cite{HL} states that $\SLhn\cong\AutC^0(\RF_n)$ and it is easily checked at the diagram level that this isomorphism is compatible, 
through the map $\iota_w$, with the one of Theorem \ref{th:wSLh=AutC}.  
Therefore if $L_1, L_2 \in \SLn$ are $\sv$--equivalent, then they represent the same automorphism in $\AutC(\RF_n)$ which is actually in $\AutC^0(\RF_n)$ 
since it corresponds to some usual string links.  According to Theorem 1.7 of \cite{HL}, this implies that $L_1$ and $L_2$ are $\cc$-equivalent.
\end{proof}

We will see below that the $\cc$--equivalence, on the other hand, does not allow such generalizations, 
and hence appears not to be the right notion to be considered in this context. 

\subsection{$\cc$--equivalence for virtual and welded string links}

In this section, we compare virtual and welded pure braids and string links up  to $\cc$-equivalence. 

In the case $n=1$, the situation is rather simple and well-known. \\
Obviously, we have $\vPh_1\cong\wPh_1\cong\{1\}$, since the virtual and welded braid groups on one strand themselves are trivial. 
In the string link case, however, virtual and welded objects differ: 
\begin{lemme}\label{factpipo}
Self-crossing change is an unknotting operation for welded $1$-string links, but isn't for virtual $1$-string links.  
In other words, we have $\wSLh_1\cong\{1\}$, whereas $\vSLh_1\ncong\{1\}$. 
\end{lemme}
\begin{proof}
Let us prove the first assertion.  A welded string link on one strand
has a Gauss diagram consisting of a single vertical strand and several
signed self-arrows,  and a crossing change on this welded long knot
corresponds to switching both the sign and orientation of one arrow.
So, for any two arrow ends that are adjacent on the vertical strand,
we may safely assume up to crossing changes that these are two arrow
tails, hence we may freely exchange their relative positions on the
strand using TC.  This implies that, up to crossing change and TC, any
Gauss diagram of a long knot can be turned into a diagram consisting
of only arrows with adjacent endpoints. By R1, such a Gauss diagram is clearly trivial. \\
We now turn to the virtual case.  Consider the virtual long knot $K_0$ shown in Figure \ref{fig:Kishino}, 
which is a string link version of the Kishino knot.  
As pointed out in \cite{DK}, the closure of $K_0$ is a virtual knot which cannot be unknotted by crossing changes. 
This proves that $K_0$ is not $\cc$-equivalent to the trivial long knot.  
\begin{figure}[h!]
\[ \xymatrix{\dessin{3cm}{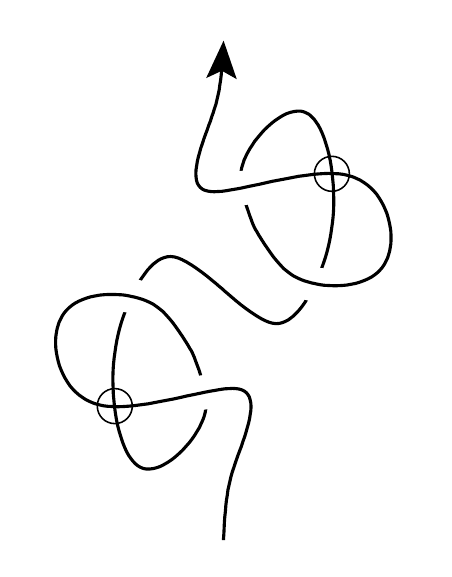}\ar@{^<-_>}[r]&\dessin{2.5cm}{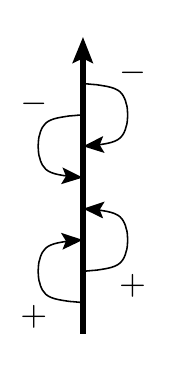}}\]
\caption{The virtual long knot $K_0$, and its Gauss diagram.} \label{fig:Kishino}
\end{figure}
\end{proof}
\noindent Similar results in the knot case can be found in \cite[Sec.~1]{DK}; note however that, unlike in the usual case, 
the closure map from virtual/welded long knots to virtual/welded knots is not an isomorphism \cite{Kauffman}.  
\begin{remarque}
As was pointed out to the authors by Anne Isabel Gaudreau, answering a question raised in a preliminary version of this paper, 
the closure map used in the latter part of the proof, from $\cc$-equivalence classes of virtual long knots to virtual knots up to self-crossing changes, 
has a non trivial kernel. 
Consider, for example, the virtual long knot $K$ represented in Figure \ref{fig:sl1}. 
On one hand, we clearly have that the closure of $K$ is trivial. 
On the other hand, the writhe polynomial defined in \cite{CG} detects $K$, showing that it is a non trivial element in $\vSLh_1$; 
indeed, it is shown in  \cite[Thm.~4.1]{CG} that the writhe polynomial is an invariant of flat virtual long knot, that is, 
an invariant of virtual long knots up to $\cc$-equivalence (see Remark \ref{rem:flat}).  
\end{remarque}

\begin{lemme}\label{prop:wSLh}
 For $n\ge 1$, there are distinct virtual and welded string links which are $\cc$--equivalent, \ie the canonical projections 
 $\xymatrix{\vSLn\ar@{->>}[r]&\vSLhn}$
 and 
 $\xymatrix{\wSLn\ar@{->>}[r]&\wSLhn}$
 are not injective.  
\end{lemme}
\begin{proof}
Consider the (usual) string link $T$ whose closure is the right-handed trefoil. 
As recalled in the introduction, the knot group is an invariant of both virtual and welded knots, which implies that $T$ is non trivial in both $\vSL_1$ and $\wSL_1$ 
(since the knot group of the trefoil is non trivial). 
However, $T$ is clearly $\cc$--equivalent to the trivial $1$-component string link. 
\end{proof}

\begin{figure}[h!]
\[
\begin{array}{ccccc}
  \xymatrix{\dessin{3cm}{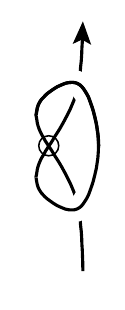}\ar@{^<-_>}[r]&\dessin{2.5cm}{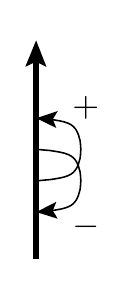}}
\end{array}
\]
\caption{The virtual string link $K$, and its Gauss diagram.} \label{fig:sl1} 
\end{figure}

Recall from \cite{Gold} that the canonical projection from the pure braid group to $\Phn$ is not injective.   The proof of Goldsmith actually applies to the virtual and welded context: 
\begin{lemme}\label{prop:wPh}
 For $n>2$, there are distinct virtual and welded pure braids which are $\cc$--equivalent,  \ie the canonical projections 
 $\xymatrix{\vPn\ar@{->>}[r]&\vPhn}$
and
 $\xymatrix{\wPn\ar@{->>}[r]&\wPhn}$
are not injective.  
\end{lemme}
\begin{proof}
Consider the pure braids $G$ and $G'$ shown in Figure \ref{fig:Gold},  which (implicitly) appear in Figure 2 of \cite{Gold}. 
\begin{figure}[h!]
\[
\begin{array}{ccc}
 \xymatrix{\dessin{2.5cm}{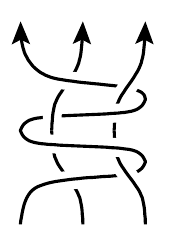}\ar@{^<-_>}[r]&\dessin{3cm}{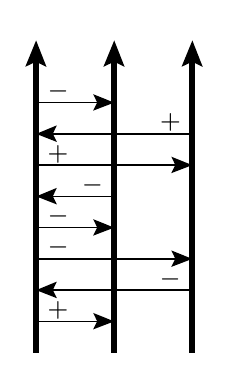}}
  &\hspace{.7cm}&
\xymatrix{\dessin{2.5cm}{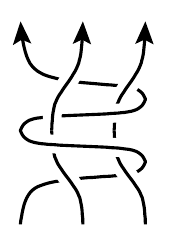}\ar@{^<-_>}[r]&\dessin{3cm}{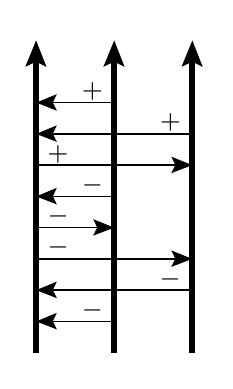}}\\
G && G'
\end{array}
\]
\caption{The virtual pure braids $G$ and $G'$, and their Gauss diagrams.} \label{fig:Gold}
\end{figure}
As shown there, these two pure braids are $\cc$--equivalent.  The result then follows by noting that usual braids embed injectively in $\vPn$ and in $\wPn$.  Indeed, the latter inclusion follows immediately from the interpretations of $\P_n$ and $\wPn$  in terms of automorphism groups of $\Fn$  (see the left rectangle in Figure \ref{fig:Connections}), and clearly implies the former. 
\end{proof}

More strikingly, although any $1$--component welded string link can be unknotted using crossing changes, this cannot always be achieved simultaneously for all strands of a welded string link with two or more components:
\begin{lemme}\label{prop:PSLh} 
For all $n>1$, there are virtual and welded string links which are not $\cc$--equivalent to any welded pure braid, \ie the inclusions $\xymatrix{\vPhn\ar@{^{(}->} [r]&\vSLhn}$ and  $\xymatrix{\wPhn\ar@{^{(}->} [r]&\wSLhn}$ are not surjective.  
\end{lemme} 
\begin{proof}
Consider the virtual $2$-string link $L$ of Figure \ref{fig:sl2} and the invariant $S_2: \vSL_2\rightarrow \Z$ defined by 
\[
S_2=\left\langle \DDBa - \DDBc - \DDCb + \DDCc - \DDDb  + \DDDc, - \right\rangle.
\]
\begin{figure}[!]
\[
\begin{array}{ccc}
   \xymatrix{\dessin{3cm}{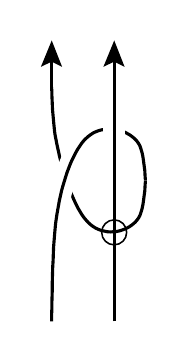}\ar@{^<-_>}[r]&\dessin{2.2cm}{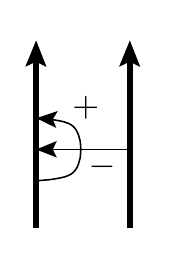}}
   &\hspace{.7cm}&
   \xymatrix{\dessin{3cm}{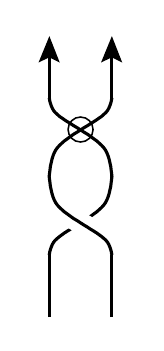}\ar@{^<-_>}[r]&\dessin{2.2cm}{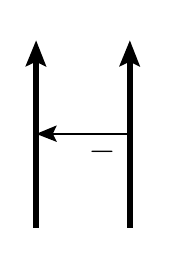}}\\
   L && B
\end{array}
\]
\caption{The virtual $2$-string links $L$ and $B$, and their Gauss diagrams.} \label{fig:sl2}
\end{figure}
Notice that $S_2$ is an invariant of $\cc$-equivalence by Theorem \ref{thm:mortier}.  Note also
that $S_2$ does detect $L$, since $S_2(L)=1$. Now assume that $L$
is $\cc$-equivalent to a pure braid. Since a pure braid
admits a representative whose Gauss diagram has only horizontal
arrows, and since the defining formula for $S_2$ contains no such diagram with only horizontal arrows,
we would have $S_2(L)=0$.  This proves that $L$ is not $\cc$-equivalent to a pure braid.
The result in the welded case follows by noting that $S_2$ is also a welded invariant, using Theorem \ref{thm:mortier}.
\end{proof}

\subsection{Comparing $\cc$ and $\sv$--equivalences}

We now compare the $\cc$-equivalence and the $\sv$-equiva\-len\-ce for welded knotted objects.  

The $1$-component case is again rather trivial. 
As seen in Lemma \ref{factpipo}, the $\cc$-equivalence yields different quotients on $\vSL_1$ and $\wSL_1$. 
The $\sv$-equivalence, on the other hand, trivializes both: $\vSLH_1\cong \wSLH_1\cong\{1\}$. 
Indeed, virtualizing all crossings of a welded (or virtual) long knot always yields the trivial element.

For $n>1$, the situation is different:
\begin{lemme} \label{prop:wSLhH}
For all $n>1$, there are virtual and welded string links which are $\sv$--equivalent but not $\cc$-equivalent,  \ie the canonical projections 
 $\xymatrix{\vSLhn\ar@{->>}[r]&\vSLHn}$
and
 $\xymatrix{\wSLhn\ar@{->>}[r]&\wSLHn}$
 are not injective.  
\end{lemme} 
\begin{proof}
Consider the welded $2$-string links $L$ and $B$ in Figure \ref{fig:sl2}. 
As shown in the proof of Lemma \ref{prop:PSLh}, $L$ is not $\cc$-equivalent to a pure braid.  
However, $L$ is equivalent, up to self-virtualization to the pure braid $B$.  
As for Lemma \ref{prop:PSLh}, the argument applies to both the virtual and welded context.  
\end{proof}

This remains true when restricting to pure braid groups:
\begin{lemme}\label{purebraids}
For all $n>1$, there are virtual and welded pure braids which are $\sv$--equivalent but not $\cc$--equivalent,  
\ie the surjective maps 
 $\xymatrix{\vPhn\ar@{->>}[r]&\vPHn}$
and
 $\xymatrix{\wPhn\ar@{->>}[r]&\wPHn}$
 are not injective.  
\end{lemme} 
\begin{proof}
Let $T$ and $T'$ be the welded pure braids shown in Figure \ref{fig:pure}.
 \begin{figure}[!h]
\[
\begin{array}{ccc}
  \xymatrix{\dessin{3cm}{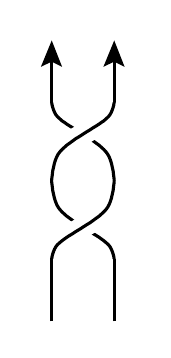}\ar@{^<-_>}[r]&\dessin{2.2cm}{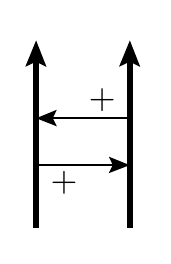}}
  &\hspace{.7cm}&
    \xymatrix{\dessin{3cm}{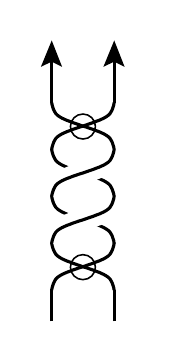}\ar@{^<-_>}[r]&\dessin{2.2cm}{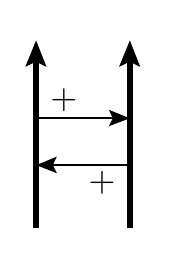}}\\
  T && T'
\end{array}
\]
\caption{The welded pure braids $T$ and $T'$, and their Gauss diagrams.} \label{fig:pure}
\end{figure} 
On one hand, we have that $T$ and $T'$ are $\sv$--equivalent. 
This is shown in Figure \ref{fig:TTb} below; 
in this figure, the first move is achieved by a sequence of classical Reidemeister moves, 
the second is a pair of self-virtualizations, 
the third is a pair of detour moves, and the final move is a planar isotopy.
 \begin{figure}[!h]
\[
\begin{array}{ccc}
  \xymatrix{\dessin{2.5cm}{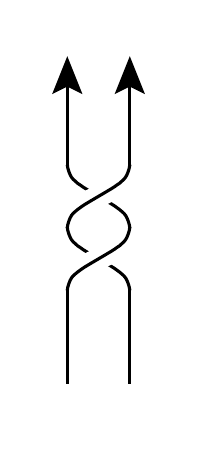}\ar@{<->}[r]^{\textrm{}}&\dessin{2.5cm}{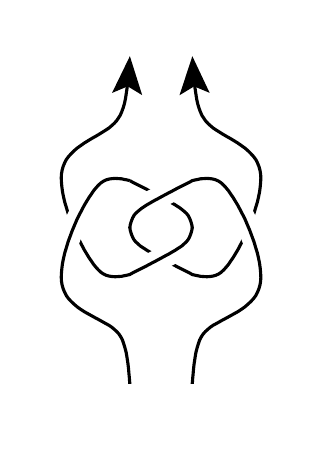}\ar@{<->}[r]^{\textrm{}}&\dessin{2.5cm}{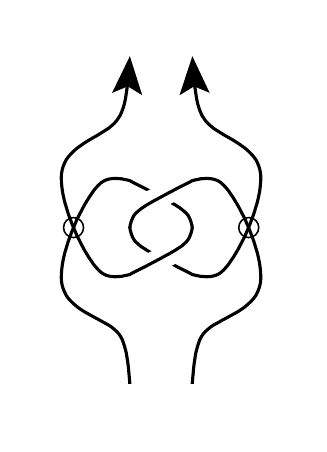}\ar@{<->}[r]^{\textrm{}}&\dessin{2.5cm}{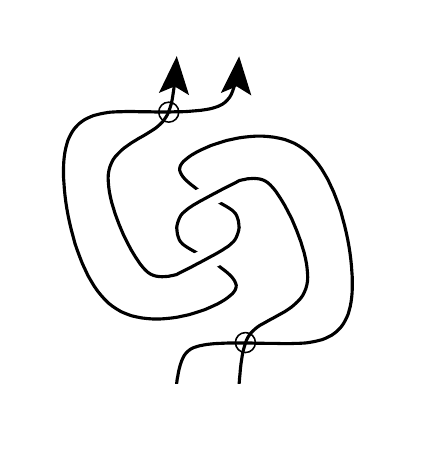}\ar@{<->}[r]^{\textrm{}}&\dessin{2.5cm}{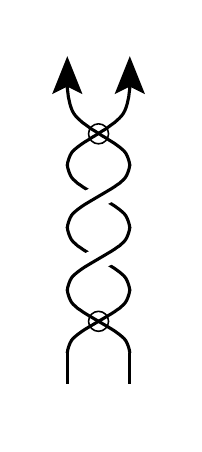}}
\end{array}
\]
\caption{The welded pure braids $T$ and $T'$ are $\sv$-equivalent.} \label{fig:TTb}
\end{figure} 
Note that, at the Gauss diagram level, this is merely an instance of a more general result on commutation of arrows supported by two strands, stated in \cite[Prop.~4.11]{vaskho}.  

On the other hand, $T$ and $T'$ are not $\cc$-equivalent.
This can be checked using the invariant $Q_2$ defined by the formula
\[
 Q_2=\left\langle \DDAa - \DDAb + \DDCb - \DDCc + \DDDc  - \DDDb, - \right\rangle. 
\]
\noi By Theorem \ref{thm:mortier}, we have that $Q_2$ is an invariant of welded $2$-string links up to $\cc$-equivalence,  
and it is straightforwardly checked that $Q_2(T)=1$, while $Q_2(T')=-1$. 
\end{proof}

Lemma \ref{prop:wSLh} readily implies that the canonical projections 
$\vSLn\to \vSLHn$ and $\wSLn\to \wSLHn$ 
 are not injective for $n\ge 1$, and   
the same observation holds for $\vPn\to \vPHn$ and $\wPn\to \wPHn$ 
by Lemma \ref{prop:wPh}, for $n>2$.   
Actually, in the welded pure braid case, this remains true for $n=2$: 
\begin{lemme}\label{pipolino}
We have $\wP_2\ncong \wPH_2$. 
\end{lemme}
\begin{proof}
One can easily prove that any automorphism in $\Aut_C(\RF_2)$ can be written as $\xi_{\eta_1,\eta_2}$
for some $\eta_1, \eta_2 \in \N$, where $\xi_{\eta_1,\eta_2} (x_1)=x_2^{\eta_1} x_1 x_2^{-\eta_1}$
and $\xi_{\eta_1,\eta_2}(x_2)= x_1^{\eta_2} x_2 x_1^{-\eta_2}$, and that $\xi_{\eta_1,\eta_2} \xi_{\eta_3,\eta_4}
=\xi_{\eta_1+\eta_3,\eta_2+\eta_4}$. 
This implies that $\wPH_2\cong \Aut_C(\RF_2)\cong \Z^2$, while it is well-known that  $\wP_2=\F_2$ (see for instance \cite{FRR}).
\end{proof}
\begin{remarque}
Lemma \ref{pipolino} can also be proved using the invariant $Q_2$ used for Lemma \ref{purebraids}. 
Indeed, it already follows from \cite[Prop.~4.11]{vaskho} that $\wPH_2$ is abelian, so it suffices to show that this is not the case for $\wP_2$. 
This is a consequence of the fact that $Q_2$ distinguishes the welded pure braids $T$ and $T'$ of Figure \ref{fig:pure}. 
More generally, since $Q_2$ is an invariant of $\cc$-equivalence, we have from this observation that none of $\vP_2$, $\vPh_2$, $\wP_2$ and $\wPh_2$ is abelian.
\end{remarque}

\subsection{Comparing virtual and welded objects}

As pointed out in \cite[Sec.~3.1]{WKO1}, welded long knots are strictly weaker than virtual long knots, 
in the sense that there exist non trivial virtual long knots which are trivial up to OC. 
This implies that, more generally, we have 
\begin{lemme}\label{lem:vwSL}
For all $n\geq1$, there are distinct virtual string links which are
equal as welded string links, \ie $\xymatrix{\vSLn\ar@{->>}[r]^(.41){\ncong}&\wSLn}$.  
\end{lemme}

\begin{remarque}\label{rem:knottedness}
It may be worth mentioning that knottedness of each individual component is not the only obstruction to having $\vSLn$ isomorphic to $\wSLn$. 
Actually, the following example shows that this also can't be valid among string links with trivial components. 
Let $L'$ be the virtual $2$-string link shown in Figure \ref{fig:sl2bis}, which is equivalent in $\wSL_2$ to the virtual pure braid $B'$ represented in the same figure. 
\begin{figure}[!]
\[
\begin{array}{ccc}
  \xymatrix{\dessin{3cm}{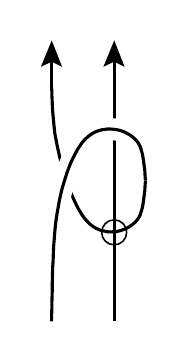}\ar@{^<-_>}[r]&\dessin{2.2cm}{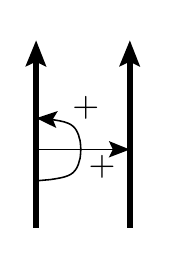}}
  &\hspace{.7cm}&
  \xymatrix{\dessin{3cm}{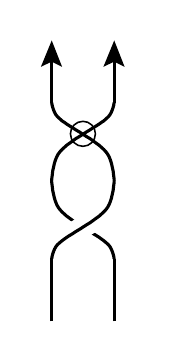}\ar@{^<-_>}[r]&\dessin{2.2cm}{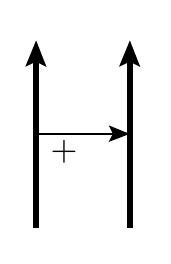}}\\
  L' && B'
\end{array}
\]
\caption{The virtual $2$-string links $L'$ and $B'$, and their Gauss diagrams.} \label{fig:sl2bis}
\end{figure}
Now, let $V_2: \vSL_2\rightarrow \Z$ be the invariant introduced in \cite{jbjktr} and defined by the Gauss diagram formula
\[
V_2=\left\langle \DDBa - \DDDa - \DDCd, - \right\rangle.
\]
We have that $V_2(L')=1$, whereas $V_2(B')=0$.

The invariant $V_2$ can moreover be symmetrized into
\[
V^*_2=\left\langle \DDBa -\DDBc - \DDDa +\DDDf - \DDCa +\DDCd, - \right\rangle.
\]
\noi which, by Theorem \ref{thm:mortier}, is invariant under self-crossing change. The same example
proves then that $\vSLhn\ncong\wSLhn$ for $n\geq2$.
\end{remarque}

By comparing the corresponding group presentations, given for instance
in \cite[Sec.~2]{WKO1}, it can be seen that the group $\vP_2$ is isomorphic to $\wP_2$. 
This remains true up to $\cc$-equivalence, i.e. we have $\vPh_2\cong \wPh_2$. 
Up to $\sv$-equivalence, the isomorphism even holds for string links, i.e. we have $\vSLH_2\cong \wSLH_2$, as a consequence of \cite[Prop.~4.11]{vaskho}.   

However, as soon as the number of strands is greater than 3, it is a general
fact that, even up to $\cc$ or $\sv$--equivalence, the virtual and
welded quotients are actually distinct: 
\begin{lemme}\label{prop:WeldedGenuine}
  For all $n>2$, there are virtual pure braids
  which are distinct, even up to $\cc$ or $\sv$-equivalence, but are
  equivalent as welded pure braids. 
\end{lemme}
 \begin{figure}[h!]
\[
\begin{array}{ccc}
  \xymatrix{\dessin{3cm}{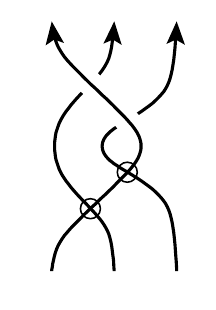}\ar@{^<-_>}[r]&\dessin{2.2cm}{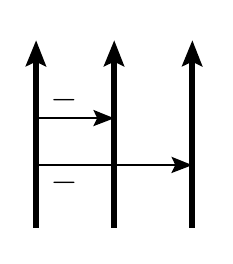}}
  &\hspace{.7cm}&
    \xymatrix{\dessin{3cm}{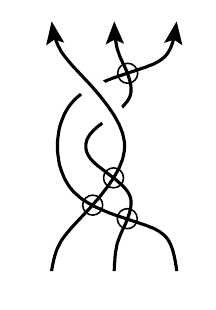}\ar@{^<-_>}[r]&\dessin{2.2cm}{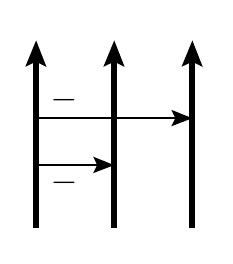}}\\
  C && C'
\end{array}
\]
\caption{The virtual pure braids $C$ and $C'$, and their Gauss diagrams.} \label{fig:CC'}
\end{figure} 
\noindent   It follows that the welded projections of
 $\vPn$, $\vPhn$, $\vPHn$, $\vSLn$, $\vSLhn$ and $\vSLHn$ are all
 non injective. 
\begin{proof}
Consider the virtual pure braids $C$ and $C'$ shown in Figure \ref{fig:CC'}. 
As obvious from the Gauss diagram point of view, they are equivalent in $\wP_3$, 
but they are distinct in $\vP_3$, $\vPh_3$ and $\vPH_3$. 
Indeed, the virtual string link invariant 
\[
M_2=\left\langle \DDEa - \DDEb - \DDEc, - \right\rangle,
\]
\noi which by Theorem \ref{thm:mortier} is invariant under self-crossing change and self-virtualization,
satisfies $M_2(C)=1$ but $M_2(C')=0$. 
\end{proof}

\section{Some open questions}
\label{sec:open-questions}

All the connections between the different notions and quotients of
string links are summarized in Figure \ref{fig:Connections}. 
However, several questions remain open, and some are listed below. 

\begin{question}\label{q:q}
Is the inclusion map $\vPn\to\vSLn$ injective ?
\end{question}
\noindent As noted in Remark \ref{rem:VirtualBraids}, this is equivalent to showing that the virtual pure braid group $\vPn$ is isomorphic to its quotient under non-monotone transformations.  
Recall that the analogous maps $\Pn\to\SLn$ and $\wPn\to\wSLn$ for usual and welded objetcs are both injective. 
It is also known that $\Pn$ embeds in $\wSLn$ (see e.g. \cite{vaskho}), which implies that $\Pn$ embeds in $\vSLn$ as well.  

\begin{question} 
Are welded pure braids the only invertibles in $\wSLn$ ? Are they in $\wSLhn$ ? 
Likewise, are virtual pure braids the only invertibles in $\vSLn$, $\vSLhn$ and $\vSLHn$ ? 
\end{question}
\noindent This is a natural question in view of the usual case, where pure braids form the group of units in $\SLn$, as shown in \cite{HL2}. 

\begin{question} 
Is the map $\vPH_n\hookrightarrow \vSLH_n$ surjective ? In other words, is any virtual string link monotone up to self--virtualization ? 
\end{question}
\noindent By Theorem \ref{th:wSLh=wPh}, the answer is affirmative in the welded case. 
Also, as stated above, we have $\vSLH_2\cong \wSLH_2\cong \Z^2$ by \cite[Prop.~4.11]{vaskho}, hence an affirmative answer in the $2$-strand case. 
In the general case, however, the answer seems likely to be negative; for example, the virtual string link depicted below is a good candidate for a counter-example, although the techniques used in the present paper do not apply:
\[
\dessin{3cm}{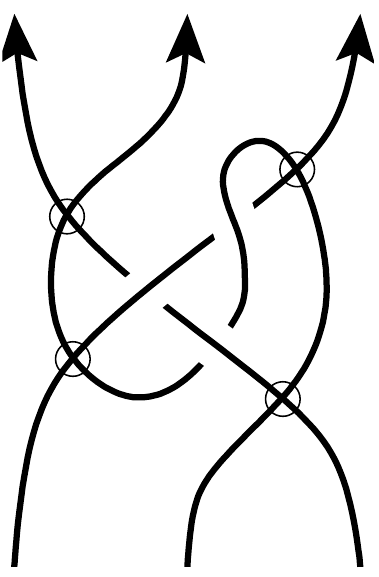}.
\]

\begin{question} 
Does $\vPh_2$ coincide with $\vP_2$ ?  Does $\wPh_2$ coincide with $\wP_2$ ?
\end{question}
 \noindent  This question is motivated by the usual case, where $\P_2$ is known to coincide with $\Ph_2$. 
 This follows from the fact that the automorphism of $\RF_2$ associated to the generator of $\Ph_2$ 
 (which is the image of the generator of $\P_2=\Z$) has no finite order.  \\

Finally, recall from Remark \ref{rem:topo} that one of the main features of welded knotted objects is that they are realized topologically as ribbon $2$-knotted objects in $4$-space, 
via Satoh's Tube map. Although less central in the present note, the following question seems worth adding. 
\begin{question} 
Is the Tube map, from welded string links to ribbon tubes \cite{vaskho}, injective ?
\end{question}
\noindent This question is in general open, see \cite{WKO1,WKO2}. 
It is true when restricting to welded braids by \cite{BH} and to string links up to self-virtualization \cite{vaskho}, but fails in the case of welded knots \cite{IK}.  
        
\bibliographystyle{abbrv}
\bibliography{wSL}

\begin{thebibliography}{10}

\bibitem{artin}
E.~{Artin}.
\newblock Theorie der z\"opfe.
\newblock {\em {Abh. Math. Semin. Univ. Hamb.}}, 4:47--72, 1926.

\bibitem{vaskho}
B.~Audoux, P.~Bellingeri, J.-B. Meilhan, and E.~Wagner.
\newblock Homotopy classification of ribbon tubes and welded string links.
\newblock arXiv e-prints:1407.0184, 2014.
\newblock To appear in Ann. Sc. Norm. Super. Pisa Cl. Sci.

\bibitem{Hoo}
D.~Bar-Natan.
\newblock Balloons and hoops and their universal finite-type invariant, {BF}
  theory, and an ultimate {A}lexander invariant.
\newblock {\em Acta Math. Vietnam.}, 40(2):271--329, 2015.

\bibitem{WKO2}
D.~Bar-Natan and Z.~Dancso.
\newblock Finite-type invariants of w-knotted objects, {II}: Tangles, foams and
  the kashiwara-vergne.
\newblock arXiv e-prints:1405.1955, 2014.
\newblock To appear in Math. Ann.

\bibitem{WKO1}
D.~Bar-Natan and Z.~Dancso.
\newblock Finite-type invariants of w-knotted objects, {I}: w-knots and the
  {A}lexander polynomial.
\newblock {\em Algebr. Geom. Topol.}, 16(2):1063--1133, 2016.

\bibitem{Bardakov}
V.~G. Bardakov.
\newblock The virtual and universal braids.
\newblock {\em Fund. Math.}, 184:1--18, 2004.

\bibitem{BH}
T.~E. Brendle and A.~Hatcher.
\newblock Configuration spaces of rings and wickets.
\newblock {\em Comment. Math. Helv.}, 88(1):131--162, 2013.

\bibitem{CKS}
J.~{Carter}, S.~{Kamada}, and M.~{Saito}.
\newblock {Stable equivalence of knots on surfaces and virtual knot
  cobordisms.}
\newblock {\em {J. Knot Theory Ramifications}}, 11(3):311--322, 2002.

\bibitem{CG}
Z.~Cheng and H.~Gao.
\newblock {A polynomial invariant of virtual links.}
\newblock {\em J. Knot Theory Ramifications}, 22(12), 2013.

\bibitem{DK}
H.~A. Dye and L.~H. Kauffman.
\newblock Virtual homotopy.
\newblock {\em J. Knot Theory Ramifications}, 19(7):935--960, 2010.

\bibitem{FRR}
R.~Fenn, R.~Rim{\'a}nyi, and C.~Rourke.
\newblock The braid-permutation group.
\newblock {\em Topology}, 36(1):123--135, 1997.

\bibitem{F}
T.~Fiedler.
\newblock {\em Gauss diagram invariants for knots and links}, volume 532 of
  {\em Mathematics and its Applications}.
\newblock Kluwer Academic Publishers, Dordrecht, 2001.

\bibitem{Gold}
D.~L. Goldsmith.
\newblock {Homotopy of braids - an answer to a question of E. Artin.}
\newblock {Topology Conf., Virginia polytechnic Inst. and State Univ. 1973,
  Lect. Notes Math. 375, 91-96 (1974).}, 1974.

\bibitem{GPV}
M.~Goussarov, M.~Polyak, and O.~Viro.
\newblock Finite type invariants of virtual and classical knots.
\newblock {\em Topology}, 39:1045--1168, 2000.

\bibitem{HL}
N.~Habegger and X.-S. Lin.
\newblock The classification of links up to link-homotopy.
\newblock {\em J. Amer. Math. Soc.}, 3:389--419, 1990.

\bibitem{HL2}
N.~Habegger and X.-S. Lin.
\newblock On link concordance and {M}ilnor's {$\overline {\mu}$} invariants.
\newblock {\em Bull. London Math. Soc.}, 30(4):419--428, 1998.

\bibitem{IK}
A.~Ichimori and T.~Kanenobu.
\newblock Ribbon torus knots presented by virtual knots with up to four
  crossings.
\newblock {\em Journal of Knot Theory and Its Ramifications}, 21(13):1240005,
  2012.

\bibitem{Kauffman}
L.~H. Kauffman.
\newblock Virtual knot theory.
\newblock {\em European J. Combin.}, 20(7):663--690, 1999.

\bibitem{Kuperberg}
G.~{Kuperberg}.
\newblock {What is a virtual link?}
\newblock {\em {Algebr. Geom. Topol.}}, 3:587--591, 2003.

\bibitem{jbjktr}
J.-B. Meilhan.
\newblock On {V}assiliev invariants of order two for string links.
\newblock {\em J. Knot Theory Ram.}, 14(5):665--687, 2005.

\bibitem{Milnor}
J.~Milnor.
\newblock Link groups.
\newblock {\em Ann. of Math. (2)}, 59:177--195, 1954.

\bibitem{mortier}
A.~Mortier.
\newblock Polyak type equations for virtual arrow diagram invariants in the
  annulus.
\newblock {\em J. Knot Theory Ramifications}, 22(7):1350034, 21, 2013.

\bibitem{polyak_arrow}
M.~Polyak.
\newblock On the algebra of arrow diagrams.
\newblock {\em Lett. Math. Phys.}, 51(4):275--291, 2000.

\bibitem{PV}
M.~Polyak and O.~Viro.
\newblock Gauss diagram formulas for {V}assiliev invariants.
\newblock {\em Internat. Math. Res. Notices}, (11):445ff., 8 pp., 1994.

\bibitem{Citare}
S.~Satoh.
\newblock Virtual knot presentation of ribbon torus-knots.
\newblock {\em J. Knot Theory Ramifications}, 9(4):531--542, 2000.

\bibitem{yajima}
T.~Yajima.
\newblock On the fundamental groups of knotted {$2$}-manifolds in the
  {$4$}-space.
\newblock {\em J. Math. Osaka City Univ.}, 13:63--71, 1962.

\end{thebibliography}

\end{document}